\documentclass{article}

\title{Almost split sequences for Kn\"{o}rr lattices}

\author{Andrew Poulton}
\date{}

\usepackage{stmaryrd,amsmath,amssymb,amsthm}
\usepackage{graphicx}
\usepackage{color}
\usepackage[margin=2.8cm]{geometry}
\usepackage[all]{xy}
\usepackage{setspace}
\newtheorem{thm}{Theorem}[section]
\newtheorem{prop}[thm]{Proposition}
\newtheorem{cor}[thm]{Corollary}
\newtheorem{lem}[thm]{Lemma}
\newtheorem*{thm*}{Theorem}

\theoremstyle{definition}
\newtheorem{defn}{Definition}
\newtheorem{eg}[thm]{Example}

\theoremstyle{remark}
\newtheorem*{rmk}{Remark}

\renewcommand{\O}{\ensuremath{\mathcal{O}}}
\newcommand{\og}{\ensuremath{\mathcal{O}G}}

\renewcommand{\hom}{\ensuremath{\mbox{Hom}}}
\newcommand{\phom}{\ensuremath{\mbox{PHom}}}
\newcommand{\shom}{\ensuremath{\mbox{\underline{Hom}}}}
\newcommand{\catt}{\ensuremath{\mathcal{T}}}

\renewcommand{\mod}{\ensuremath{\mbox{\textbf{mod}}}}
\newcommand{\Mod}{\ensuremath{\mbox{\textbf{Mod}}}}
\newcommand{\modp}{\ensuremath{\mbox{\textbf{mod}}_{\mathcal{P}}}}
\newcommand{\stab}{\ensuremath{\underline{\mbox{\textbf{mod}}}}}

\newcommand{\slat}{\ensuremath{\underline{\mbox{\textbf{latt}}}}}
\newcommand{\im}{\ensuremath{\mbox{Im}}}
\newcommand{\cats}{\ensuremath{\mathcal{S}}}
\newcommand{\Z}{\ensuremath{\mathbb{Z}}}
\newcommand{\N}{\ensuremath{\mathbb{N}}}
\renewcommand{\L}{\ensuremath{\Lambda}}

\renewcommand{\ker}{\ensuremath{\mbox{Ker}}}

\newcommand{\End}{\ensuremath{\mbox{End}}}
\newcommand{\send}{\ensuremath{\underline{\mbox{End}}}}

\newcommand{\rad}{\ensuremath{\mbox{Rad}}}
\newcommand{\soc}{\ensuremath{\mbox{Soc}}}
\newcommand{\he}{\ensuremath{\mbox{Head}}}

\newcommand{\tr}{\ensuremath{\mbox{Trace}}}
\newcommand{\coker}{\ensuremath{\mbox{Coker}}}

\newcommand{\tor}{\ensuremath{\mbox{Tor}}}
\newcommand{\ok}{\ensuremath{\Omega_{kG}}}
\newcommand{\oo}{\ensuremath{\Omega_{\og}}}
\newcommand{\rk}{\ensuremath{\mbox{Rank}}}
\newcommand{\ol}{\ensuremath{\Omega_{\Lg}}}

\renewcommand{\r}{\ensuremath{\rightarrow}}

\renewcommand{\u}{\ensuremath{\!\uparrow}}
\renewcommand{\d}{\ensuremath{\!\!\downarrow}}

\mathchardef\md="2D

\newcommand{\OG}{\ensuremath{\mod\md\og}}
\newcommand{\Lg}{\ensuremath{\Lambda G}}

\newcommand{\sG}{\ensuremath{\stab\md\og}}
\newcommand{\sGl}{\ensuremath{\stab\md\Lg}}

\newcommand{\slG}{\ensuremath{\slat\md\og}}
\newcommand{\lslG}{\ensuremath{\slat\md\Lg}}

\newcommand{\lG}{\ensuremath{\underline{\mbox{\textbf{Mod}}}}\md\Lg}

\newcommand{\ldk}{\ensuremath{D^b(\mod\md\Lg)/K^b(P_{\Lg})}}
\newcommand{\ep}{\ensuremath{\varepsilon}}
\renewcommand{\L}{\ensuremath{\Lambda}}

\begin{document}
\maketitle
\begin{abstract}
Let $\O$ be a complete discrete valuation ring with maximal ideal $(\pi)$ and residue field $k=\O/\pi\O$, $G$ a finite group, and $\og$ the corresponding group algebra.

We give necessary and sufficient conditions for the middle term of an almost split sequence terminating in Kn\"{o}rr lattice to be indecomposable (Theorem \ref{mdterm}).  The main tool we use is an adjunction in the stable category of $\og$ (Theorem \ref{kag}), which is hopefully of independent interest.  As a second application of this adjunction, we give a characterisation of the stable endomorphism rings of Heller lattices of $kG$-modules when $\O$ is ramified.  \end{abstract}
\section{Introduction}
Let $\{K,\O,k\}$ be a $p$-modular system for a finite group $G$, so that $\O$ is a complete discrete valuation ring with maximal ideal $J(\O)=(\pi)$, residue field $k=\O/\pi\O$ characteristic $p>0$ and field of fractions $K$. Let $\nu:\O\r \N$ denote the $\pi$-adic valuation on $\O$.

The present paper is concerned with presenting two applications of a functor $R:\sG\r\slG$, where $\sG$ is the stable category of the Frobenius category of $\og$-modules with respect to the $\O$-split exact structure, and $\slG$ is the full triangulated subcategory of $\sG$ consisting of $\og$-lattices.  This functor is a right adjoint to the inclusion functor $\slG\r\sG$ (see Theorem \ref{kreg}).

The first application deals with almost split sequences for Kn\"{o}rr lattices.  Recall that a \emph{Kn\"{o}rr lattice} $M$ (\emph{virtually irreducible} lattice in \cite{knorr}) is a lattice satisfying 
\[\nu(\tr(\alpha))\geq\nu(\tr(\rk_\O(M)))\]
for all $\alpha\in\End_{\og}(M)$ (with equality if and only if $\alpha$ is an automorphism), where $\tr:\End_{\O}(M)\r\O$ is the ordinary trace map.  Examples of Kn\"{o}rr lattices include irreducible lattices and absolutely indecomposable lattices of rank coprime to $p$. Further, R. Kn\"{o}rr showed that all Kn\"{o}rr lattices are absolutely indecomposable.  Recall that Carlson and Jones \cite{C-J} define the \emph{exponent} $\exp(M)$ of an $\og$-lattice $M$ to be $\pi^a$ if $\pi^aId_M$ factors through a projective module, but $\pi^{a-1}Id_M$ does not.  It is shown in \cite{C-J} that Kn\"{o}rr lattices are the absolutely indecomposable lattices $M$ such that $\pi^{a-1}Id_M$ generates the socle of $\send(M)$, the endomorphism algebra of $M$ as an object in $\slG$.

If $M$ is a Kn\"{o}rr lattice of exponent $\pi^a$ and $\Delta$ is the connected component of the Auslander-Reiten quiver of $\og$ containing $M$, then recent papers by Geline and Mazza \cite{gm} and Kawata \cite{kaw4} give sufficient conditions for $M$ to lie on the end of $\Delta$.  To be precise, \cite[Theorem 1.1]{gm} shows that $M$ is at the end of $\Delta$ provided $M/\pi M$ is indecomposable and $a>1$, whereas \cite[Theorem 3.1]{kaw4} implies that $M$ is at the end of $\Delta$ provided $\exp(M)=\exp(\O)$ where $\O$ is the trivial lattice.  We provide here necessary and sufficient conditions, in a similar vein those above.  Namely, we prove:
\begin{thm} \label{comeon}
Let $M$ and $\Delta$ be as above, and suppose $M$ has exponent $\pi^a$.  Then $M$ lies at the end of $\Delta$ if and only if $M/\pi^{a-1}M$ is indecomposable.
\end{thm}  
The case when $a=1$ follows from Proposition \ref{tooeasy}, which shows how restrictive the existence of a Kn\"{o}rr lattice of exponent $\pi$ in a given block is.  This extends a result of Jones, Kawata and Michler \cite[Theorem 2.1]{jkm}.   

The second application is to stable endomorphism rings of Heller lattices of $kG$-modules.  We remind the reader that if $X$ a $kG$-module, then the Heller lattice of $X$ is the kernel $\oo X$ of an $\og$-projective cover $P(X)\r X$.  We show the following:
\begin{thm} \label{blahblat}
Suppose $(p)\subseteq (\pi^2)$, and let $X$ be a $kG$-module.  Then $\send(\oo M)$ is isomorphic to the trivial extension algebra of $\send(X)$.
\end{thm}
Previously, Kawata \cite[Theorem 2.9]{kaw2} has shown that $\oo X$ is indecomposable when $X$ is non-projective indecomposable and $(p)\subseteq {\pi^3}$.  It follows from Theorem \ref{blahblat} and Proposition \ref{hirsuite} that Heller lattices of non-projective indecomposable $kG$ modules are decomposable when $(p)\subseteq (\pi^2)$.

It turns out that the existence of the functor $R$ requires the fact that $\O$ is noetherian of finite global dimension (i.e. $\O$ is \emph {regular}), and thus we obtain the following result, which is hopefully of independent interest:

\begin{thm} \label{kreg}
Let $\L$ be a regular ring, and let $\slat\md\L G$ denote the full subcategory of $\stab\md \L G$ consisting of $\L G$-modules that are projective as $\L$-modules.  Then the inclusion $\slat\md\L G\r\stab\md \L G$ admits a right adjoint $R:\stab\md \L G\r\slat\md \L G$. 
\end{thm} 



Throughout, an $\og$-module will mean a right $\og$-module and homomorphisms will act on the left, so that the composition $A\stackrel{\alpha}{\r}B\stackrel{\beta}{\r} C$ will be written $\beta \alpha$. 

\section*{Acknowledgements}
The author would like to thank his supervisor Jeremy Rickard for many patient and helpful discussions, especially with respect to Section 3.  The author is grateful for the financial support of the Engineering and Physical Sciences Research Council.  

\section{Stable categories over commutative rings}
We recall here the construction given in \cite{BIK} of the stable category of a group algebra over a commutative ring, and refer the reader to that paper for full details.  Let $G$ denote a finite group, $\L$ a commutative ring (with unity) and let $\Lg$ denote the corresponding group algebra.  When $\L$ is not a field, it is generally not the case that $\Lg$ is self-injective, and hence the usual way of constructing the stable category `modulo projectives' does not yield a triangulated category.  This is corrected by considering those modules that are projective relative to the trivial subgroup.  

An $\Lg$-module $M$ is said to be \emph{weakly injective} if $M$ is a summand of an induced module $N\otimes_{\L}\Lg$ (where $N$ is some $\L$-module).  This is equivalent to every $\L$-split monomorphism from $M$ splits over $\Lg$, and that every $\L$-split epimorphism to $M$ splits over $\Lg$.  It is shown in \cite{BIK} that for every $\Lg$-module $M$ there are natural $\L$-split monomorphisms $\iota_M:M\r M\otimes_{\L}\Lg$ and natural $\L$-split epimorphisms $\theta_M:M\otimes_{\L} \Lg \r M$.  In particular, the module category $\Mod\md\Lg$ is not only an \emph{exact category} with respect to the $\L$-split structure in the sense of Quillen \cite{qui}, but is in fact a \emph{Frobenius category}.  Happel \cite{hap} shows how to construct a triangulated `stable' category from a Frobenius category.

\begin{defn}
The \emph{stable category} $\lG$ of $\Lg$ has as its objects $\Lg$-modules, and the morphism spaces are the quotients 
 \[\hom_{\lG}(M,N):=\shom(M,N)=\hom_{\Lg}(M,N)/\phom_{\Lg}(M,N),\]
where $\phom_{\Lg}(M,N)$ is the space spanned by those morphisms factoring through a weakly injective module.  It is shown in \cite{BIK} that $\phom_{\Lg}(M,N)$ can be identified with the image $\hom_{\L}(M,N)_1^G$ of the \emph{trace map} $\mbox{Tr}_G:\hom_{\L}(M,N)\r \hom_{\Lg}(M,N)$ defined by $\mbox{Tr}_G(f)(m)=\sum_{g\in G}g(f(g^{-1}m))$ for $m\in M$.
\end{defn}
  The translation functor in $\lG$ will be denoted by $\Omega$, so that $\Omega(M)$ is isomorphic (in $\lG$) to the kernel of $\theta_M$.  If $M$ and $N$ are isomorphic in $\lG$, we denote this by $M\simeq N$.  Similarly, if $f$ and $f'$ are maps in $\Mod\md\Lg$ that are equivalent in $\lG$, we express this by $f\simeq f'$.  Given a ring $\Gamma$, we write $\Omega_{\Gamma}$ for the Heller operator on the category of $\Gamma$-modules.  We denote by $\lslG$ the full subcategory of $\lG$ consisting of finitely presented modules.

If $\Gamma$ is an arbitrary ring, we let $\modp\md\Gamma$ denote category of $\Gamma$-modules modulo projectives.  This has $\Gamma$-modules as objects, but now two morphisms $f,f'$ are identified in $\modp\md\Gamma$ if $f-f'$ factors through a projective $\Gamma$-module (provided the difference $f-f'$ makes sense).  The key reason for introducing this category, as will become apparent, is that the the Heller operator $\Omega_\Gamma$ defines an additive endofunctor on $\modp\md\Gamma$.  Further, while $\modp\md\Lg$ is not generally triangulated (just additive), since projective modules are also weakly injective, there is a canonical (additive) functor $\Psi:\modp\md\Lg\r\lG$.  

\section{An adjunction} \label{adj}
Now let $\L$ be a commutative noetherian ring with finite global dimension $n$ (so that $\L$ is \emph{regular}).  Let $\lslG$ denote the full triangulated subcategory of $\sGl$ consisting of $\Lg$ modules that are projective as $\L$-modules.  Objects in $\lslG$ will be referred to as \emph{$\Lg$-lattices}.  Since projective resolutions of $\Lg$-lattices are $\L$-split, it follows that each non-projective $\Lg$-lattice has infinite projective dimension.  Further, given an arbitrary $\Lg$-module $M$, the module $\ol^{n}M$ is an $\Lg$-lattice, since the exact sequence defining $\ol^{n}M$ is a projective resolution of $M$ as an $\L$-module.  In particular, all weakly injective modules have finite projective dimension.  

\begin{lem} \label{oboi}
Let $M$ be a lattice, and $N$ an arbitrary $\Lg$-module.  Then the map 
\[\Sigma:\hom_{\modp\md\Lg}(M,N)\tilde{\r} \hom_{\modp\md\Lg}(\ol M,\ol N)\]
induced by the Heller operator is an isomorphism (of $\L$-modules).
\end{lem}
\begin{proof}
Consider the following commutative diagram where $P$ and $Q$ are projective and the rows are exact:
\[
\xymatrix{
0\ar[r] & \ol M \ar[r]^{\alpha_M} \ar[d]^{\ol \psi} & P \ar[r]^{\pi_M}\ar[d] &  M \ar[r]\ar[d]^{\psi} & 0 \\
0\ar[r] & \ol N \ar[r]^{\alpha_N} &  Q \ar[r]^{\pi_N} &  N \ar[r] & 0.
}
\]
Here, the map $P\r Q$ is any map lifting $\psi$, and then by definition $\ol\psi$ is any map making the diagram commute.  Since $\ol$ is an endofunctor on $\modp\md\Lg$, the mapping $\psi\mapsto \ol \psi$ is well defined, so that a different lifting of $\psi$ results in a morphism equivalent to $\ol \psi $ in $\modp\md\Lg$.

We define an inverse to $\Sigma$ as follows.  As $M$ is a lattice, the top row splits over $\L$, and since $Q$ is projective (and thus weakly injective), any map $f:\ol M \r \ol N$ lifts to a map $f':P\r Q$ such that $\alpha_N\circ f=f'\circ \alpha_M$.  Then $f'$ induces a map $f'':M\r N$  such that $f''\circ\pi_M=\pi_N\circ f''$.  The assignment $\Sigma^{-1}:f\mapsto f''$ will be the inverse in question (it is clearly $\L$-linear).  We claim this is a well-defined map.  Indeed, suppose $\alpha, \beta:P\r Q$ are distinct lifts of $f$ and suppose they induce $a,b: M\r N$ respectively.  Then $(\alpha-\beta)\alpha_M=0$, so $\alpha-\beta=\gamma\pi_M$ for some $\gamma:M\r Q$.  Thus $(a-b)\pi_M=\pi_N(\alpha-\beta)=\pi_N\cdot\gamma\cdot\pi_M$.  It follows that $a-b=\pi_N\gamma$ is projective.  Thus different lifts of $f$ induce equivalent morphisms $M\r N$ in $\modp\md\Lg$, and $\Sigma^{-1}$ is well defined.  We observe immediately that $\Sigma$ is surjective.

Finally, if $\ol\psi$ is projective, we have $\ol\psi=\theta\circ \alpha_M$ for some $\theta:P\r\ol N$.  Then $\alpha_N\circ\theta:P\r Q$ lifts $\ol\psi$, and it follows (since $\Sigma^{-1}$ is well defined) that $\psi$ is therefore projective.  Hence $\ol$ is injective.
\end{proof}

Let $D^b(\mod\md\Lg)$ denote the bounded derived category of $\mod\md\Lg$, and $K^{-,b}(P_{\Lg})$ the homotopy category of bounded above complexes of finitely generated projective modules with bounded homology.  Recall that there is an equivalence $K^{-,b}(P_{\Lg})\stackrel{\sim}{\r} D^b(\mod\md\Lg)$, and via this equivalence we identify the thick subcategory of perfect complexes, $K^b(P_{\Lg})$, of $K^{-,b}(P_{\Lg})$ with a thick subcategory of $D^b(\mod\md\Lg)$.  Thus the Verdier quotient $\ldk$ exists.

We have an obvious functor.  Let $F':\mod\md\Lg\r\ldk$ be the composition of the natural functor $\mod\md\Lg\r D^b(\Lg)$ and the quotient functor.  The kernel of this functor contains precisely the $\Lg$-modules of finite projective dimension (which includes the weakly injective modules), so $F'$ factors through the stable category, yielding a functor $F:\sGl\r\ldk$.  Observe that $F$ is exact, since triangles $X\r Y\r Z \r \Omega^{-1}X$ in $\sGl$ come from pushout diagrams
\[
\xymatrix{
0 \ar[r] & X \ar[r]\ar[d] & X\otimes \Lg \ar[r]\ar[d] & \Omega^{-1}X\ar[r] \ar@{=}[d] & 0\\
0 \ar[r] & Y \ar[r] & Z \ar[r] & \Omega^{-1}X \ar[r] & 0.
}
\] 

Rickard \cite[Theorem 2.1]{rik} has shown that $F$ is an equivalence when $\L$ is a field.

\begin{lem} \label{brap}
$F$ is essentially surjective.
\end{lem}
\begin{proof}
We will show something stronger, that every object of $\ldk$ is isomorphic to $FM$ for some $\Lg$-lattice $M$.  The idea behind the proof follows \cite[Theorem 2.1]{rik}.
Let $X=X^*$ be an object of $D^b(\Lg)$. We have that, up to isomorphism in $D^b(\mod\md\Lg)$, 
\[X^*= \cdots \r P^r\r P^{r+1} \r\cdots\r P^s\r0,\]
where the $P^i$ are projective, and $H^i(X)=0$ for $i< r<0$.  The mapping cone of the natural map from $X$ to
\[\tilde{X}= \cdots P^{r-n-1}\r P^{r-n} \r 0\r \cdots\]
is a bounded complex of projectives, and hence $X\cong\tilde{X}$ in $\ldk$.  Since the cokernel $M$ of $P^{r-n-1}\r P^{r-n}$ is the $n$th kernel in a projective resolution of $\coker([P^{r-2}\r P^{r-1}])$, it follows that $M$ is a lattice.  Since $\tilde{X}$ has homology concentrated in degree $r-n$ isomorphic to $M$, it follows that $\tilde{X}[n-r]$ is isomorphic in $\ldk$ to $F(M)$.  Since $F$ is exact, we conclude that $X\cong \tilde{X} \cong F(\Omega^{r-n}(M))$.
 
\end{proof}


The following lemma can be essentially found in \cite{rik}.
\begin{lem} \label{bit}
Let $F:\cats\r\catt$ be a full, exact functor between triangulated categories such that for any non-zero $X\in \cats$, we have $FX\neq 0$.  Then $F$ is faithful.
\end{lem}
\begin{proof}
Let $X \stackrel{\alpha}{\r} Y$ be such that $F\alpha=0$.  We embed $\alpha$ in a triangle
\[X\stackrel{\alpha}{\r} Y \stackrel{\gamma}{\r} Z \r,\]
and obtain a triangle
\[FX\stackrel{0}{\r} FY \stackrel{F\gamma}{\r} FZ \r.\]
Since $FZ=\mbox{cone}(0)$, we have $FZ\cong FY\oplus FX[1]$, and hence a map $F\beta: FZ \r FY$ (the image under $F$ of $\beta:Z\r Y$, which exists as $F$ is full) such that $F\beta\circ F\gamma=Id_{FY}$.  Then $\beta\circ \gamma:Y\r Y$ has its mapping cone sent to $0$ by $F$, so by hypothesis $\mbox{cone}(\beta\circ\gamma)=0$. Then $\beta\circ\gamma$ is an isomorphism, so $\gamma$ is a split monomorphism, and hence $\alpha=0$.
\end{proof}

Let $L:\lslG\r \sGl$ denote the inclusion functor.  The following proposition is a generalization of \cite[Theorem 2.1]{rik}.

\begin{prop} \label{eq}
$FL:\lslG\r\ldk$ is an equivalence of triangulated categories.
\end{prop}  
\begin{proof}
First note that $FL$ is full, 
and since every non-projective lattice has infinite projective dimension, for no non-projective lattice $X$ is $FL(X)$ weakly injective.  By Lemma \ref{bit}, $FL$ is fully faithful.  Finally, $FL$ is essentially surjective by Lemma \ref{brap}.
\end{proof}


From this equivalence we recover a functor \[R:\sGl \stackrel{F}{\r} \ldk \simeq \lslG.\]

\begin{thm}\label{kag}
$R$ is a right adjoint to $L$.
\end{thm}
\begin{proof} 
Let $M$ be a lattice, $X$ an $\Lg$-module, and define \[\psi=\psi_{M,X}:=\hom_{\sGl}(LM,X)\r\hom_{\lslG}(M,RX)\]
by $\psi(f)=\Omega^{-n}\ol^n f$.  We first check that $\psi$ is well defined, and note that naturality follows easily.  Since $\Omega$ is a self-equivalence of $\sGl$, we need only check that $\ol^n f$ is well-defined for $f\in\hom_{\sGl}(LM,X)$.  In $D^b(\Lg)$, $f$ induces the following well-defined map
\[
\xymatrix{
\cdots\ar[r]  & P_i \ar[r]\ar[d]^{f'_i} & \cdots \ar[r] & P_0\ar[r] \ar[d]^{f'_0}  & 0 \\
\cdots\ar[r]  &  Q_{i} \ar[r] &  \cdots \ar[r] & Q_0 \ar[r]  & 0,
}
\]
where the $P_i$ form a relative projective resolution of $M$, and the $Q_i$ a projective resolution of $X$, and the $f'_i$ are the components of the chain map lifting $f$.  We have induced maps $\ol^{i}f:\ol^iM\r\ol^iX$ for each $i$.  To check that $\psi$ is well-defined, we therefore need only show that if $f$ is zero (i.e. projective), then $\ol^nf$ is projective.  However, this is implied by Lemma \ref{oboi} since $M$ is a lattice.

Now we show that $\psi_{M,X}$ is an isomorphism for all $M,X$. Note that $f$ induces the following commutative diagram
\[
\xymatrix{
0\ar[r] & \ol^n M \ar[r] \ar[d]^{\ol^n f} & P_{n-1} \ar[r] \ar[d] & \ol^{n-1} M \ar[r]\ar[d]^{\ol^{n-1}f} & 0 \\
0\ar[r] & \ol^{n}  X \ar[r]               & Q_{n-1} \ar[r]        & \ol^{n-1} X \ar[r]                   & 0,
}
\]
and so Lemma \ref{oboi} shows that $\ol^{n-1}f$ is projective if (and only if) $\ol^n f$ is.  By the same argument, we find that all the $\ol^{i}f$ are not projective, and hence $f$ is not projective.  and $\psi$ is injective.


Let $\alpha\in\hom_{\lslG}(M,RX)=\hom_{\lslG}(RM,RX)$.  Then we have a commutative diagram  
\[
\xymatrix{
0\ar[r] & \ol^n M \ar[r] \ar[d]^{\Omega^n\alpha} & \ol^n M\u \ar[r] \ar[d]&\cdots\ar[r] & \Omega^{-n}\ol^n M\u \ar[r]\ar[d] & \Omega^{-n}\ol^n M \ar[r]\ar[d]^\alpha & 0 \\
0\ar[r] & \ol^n X \ar[r] & \ol^n X\u \ar[r] &\cdots \ar[r] & \Omega^{-n}\ol^n X \u \ar[r] & \Omega^{-n}\ol^n X \ar[r] & 0
}
\]
yielding a diagram
\[
\xymatrix{
0\ar[r] & \ol^n M \ar[r] \ar[d]^{\Omega^n\alpha} & P_{n-1} \ar[r] \ar[d]&\cdots\ar[r] & P_0 \ar[r]\ar[d] &  M \ar[r]\ar[d]^\beta & 0 \\
0\ar[r] & \ol^n X \ar[r] & Q_{n-1} \ar[r] &\cdots \ar[r] & Q_0  \ar[r] & X \ar[r] & 0.
}
\]
By repeated applications of Lemma \ref{oboi}, $\beta$ is non-projective and well-defined by $\ol^n f$.  It follows that $\alpha=\psi(\beta)$, and $\psi$ is surjective.
\end{proof}

\begin{rmk}
We will habitually identify $M$ with $LM$ in $\sGl$ when $M$ is a lattice, and we will write $\ep_M:RM\r M$ for the counit of the adjoint pair $(L,R)$.
\end{rmk}

Let $\ker R$ denote the kernel of $R$, the full subcategory of $\sGl$ consisting of objects $X$ with $RX\simeq 0$.  This is a thick subcategory of $\sG$ and its objects are precisely the modules of finite projective dimension.  It is not hard to construct $\Lg$-modules of finite projective dimension that are not weakly injective.  An obvious source are the cones of the counit maps, but it is even possible to find torsion modules in the kernel.

\begin{eg} \label{gtytg}
Let $G=C_2=\langle g \mid g^2=1 \rangle$, and let $\L=\Z_2$, the $2$-adic integers.  Let $\alpha:\Lg\r\Lg$ be given by the matrix $\begin{pmatrix}1 & 3 \\ 3 & 1\end{pmatrix}$ (with respect to the basis $(1_G,g)$).  Note that $\det(\alpha)=-8$, so $\alpha$ is injective.  The Smith normal form of $\alpha$ is $\begin{pmatrix} 1 & 0 \\ 0 & 8 \end{pmatrix}$, and so $\coker (\alpha)$ is $\L/8\L$ as an $\L$-module (the action of $g$ can be checked to be multiplication by $5$).  Since $\coker (\alpha)$ is cyclic, it cannot be weakly injective, yet by construction has projective dimension 1.
\end{eg}

It follows from this example that the kernel of $R$ is not generally trivial.  We can characterize the kernel in terms of the counit of the adjunction.  Recall that if $C$ is a subcategory of a triangulated category $\mathcal{T}$, then $C^{\perp}=\{T\in \mathcal{T}\mid \hom_{\mathcal{T}}(M,T)=0 \mbox{ \  for all  \ } M\in C\}$.
\begin{lem} \label{kerlem}
An $\Lg$-module $X$ is in the kernel of $R$ if and only if $\ep_X\simeq 0$.  In particular, $\ker R= \slG^\perp$.
\end{lem}
\begin{proof} 
It is clear that the kernel of $R$ consists of those modules $X$ of finite projective dimension, in which case the counit is clearly a weakly injective map.  Conversely, the construction of the counit $\ep_X:RX \r X$ shows that we have $\ep_X\simeq\ol^{-n}(1_{\ol^{n}X})$.  Hence if $\ep_X$ is weakly injective, the identity of $\ol^nX$ is weakly injective, and so $X$ has finite projective dimension. 
The final statement is clear.
\end{proof}

In \cite{BIK}, the authors note that, if $k$ is a quotient ring of $\L$ such that $k$ is finitely presented as an $\L$-module, then the full subcategory $D=\{Y\in \sGl \mid Y\cong Y\otimes k\}$ admits $-\otimes k$ as a left adjoint to the inclusion $D\r \sGl$, i.e. if $M\in D$ and $X$ is an arbitrary object in $\sGl$, then $\shom(X,M)\cong \shom(X\otimes k,M)$.  This holds in particular if $\L$ is noetherian and $k$ is a quotient field of $\L$.  

We have seen that the kernel of $R$ is generally non-trivial.  However, the following shows that when $\L$ is noetherian, every non-projective $kG$-module has infinite projective dimension as an $\Lg$-module when $k$ is a quotient field of $\L$ with characteristic dividing $|G|$ (note that if the characteristic of $k$ does not divide $|G|$, then every $kG$-module is weakly injective, and hence has finite projective dimension as an $\Lg$-module).  
\begin{prop} \label{kelm}
Let $\L$ be a regular ring of global dimension $n$, and suppose $k$ is a quotient field of $\L$ whose characteristic divides $|G|$.  Let $M$ be a non-projective $kG$-module.  Then $RM\not\simeq 0$.
\end{prop}
\begin{proof}
By Lemma \ref{kerlem}, $RM\simeq 0$ is equivalent to $\ol^n(M)$ being weakly injective, which in turn is equivalent to $\ol^n(M)\d^G_S$ being weakly injective for $S$ a Sylow $p$-subgroup of $G$ (where $p=\mbox{Char}(k)$).  Since restriction is an exact functor, we may assume $G$ is a $p$-group.  Then, for every non-projective $kG$-module $M$ we have
\[\shom(\L,M)\cong \shom(k,M)\neq 0.\]
Thus $RM\not\simeq 0$ by Lemma \ref{kerlem}. 
\end{proof}

\section{Almost split sequences for Kn\"{o}rr lattices}
From now on, let $\O$ denote a complete discrete valuation ring with maximal ideal $J(\O)=\mathfrak{m}$ principly generated by some $\pi\in \O$.  To avoid trivialities, we assume $|G|$ is not a unit in $\O$.  This means that the residue field $k=\O/\mathfrak{m}$ has characteristic dividing the order of $|G|$, and hence $(p)=(\pi)^e$ for some $e\geq1$, where $p$ is the characteristic of $k$.  It is well-known that the Heller operators $\oo$ and $\oo^{-1}$ preserve indecomposability of $\og$-lattices, and that any surjective map from a non-projective indecomposable $\og$-lattice to an indecomposable $(\O/\pi^b\O)G$-lattice (for all $b\in \N)$ is non-zero in $\sG$.  Moreover, Thompson \cite{thomp} has shown that an $\og$ lattice $M$ has a projective summand $P$ if and only if $P\otimes k$ is a summand of $M\otimes k$.   We will use these facts without comment.  Further, we will frequently write $\overline{M}$ in place of $M\otimes k$.   

Since $\O$ is hereditary, we have that $RM\simeq \oo^{-1}(\oo(M))$ in $\slG$ (and $\sG$).  As $\oo^{-1}(P)\cong 0$ when $P$ is projective, we will take $RM$ to be projective-free by convention, when viewed as an $\og$-lattice.

Note that if $M$ is an $\og$-lattice (or $kG$-module), then $\oo M$ (or $\ok M$) is isomorphic to $\Omega M$ in $\sG$, and we will implicitly identify them.  In particular $\Omega M$ and $\Omega^{-1} M$, when considered in the module category, will be assumed to be projective-free.  Finally, given an $\og$-module $M$, we denote the $\og$-projective cover of $M$ by $P(M)$.

Let us recall the construction of almost split sequences for $\og$-lattices.  If $N$ is an indecomposable $\og$-lattice with stable $\og$-endomorphism ring $E_N$, then $E_N$ has a simple socle (as an ($E_N\md E_N)$-bimodule).  A generator of this socle is called an \emph{almost projective morphism}. It is shown in \cite[Theorem 34.11]{tev} that there is an almost split sequence
\[0\r \oo N \r Y_N \stackrel{\psi_N}{\r} N \r 0\]
in the category of $\og$-lattices where, $\psi_N:Y_N\r N$ is a pullback of an almost projective morphism along a projective cover $P_N\r N$.

We note the following immediate consequence:
\begin{prop} \label{arl}
Let $M$ be an indecomposable lattice with stable endomorphism ring $k$. Then every non-projective map to $M$ is a split surjection, and every non-projective map from $M$ is a split monomorphism.
\end{prop}
\begin{proof}
If $\send(M)=k$, then $Id_M$ is an almost projective morphism.  Therefore any projective cover $P(M)\r M$ is right almost split and any injective hull $M\r I(M)$ is left almost split.
\end{proof}

\subsection{Lattices of minimal exponent}
Recall that Carlson and Jones \cite{C-J} define the \emph{exponent} $\exp(B)$ of a torsion $\O$-module $B$ to be the least power of $\pi$ annihilating $B$.  Further, the exponent of an $\og$-module $M$ is the exponent of the $\O$-module $\send(M)$.  We write $\pi^a< \pi^b$ if $a<b$. It is a well-known property of $\og$ lattices that $\exp(\shom(M,N))\leq \exp(L) \ \ \mbox{and  } \exp(M)$ (\cite[Lemma 2.1]{C-J}). For instance, Heller lattices of non-projective $kG$-modules all have exponent $\pi$, and $\exp\shom(M,N)=\pi$ if either $M$ or $N$ is the Heller lattice of some non-projective $kG$-module.  This is a consequence of the adjunction $R$.  
Kawata \cite{kaw1} has previously noted that the indecomposable lattices of exponent $\pi$ are precisely the Heller lattices of indecomposable non-projective $kG$-modules when $\O$ is a ramified extension of ramification degree at least 3 of some complete discrete valuation ring $\O'$.  We will prove a generalisation of this result independent of ramification in $\O$, and from this and Theorem \ref{uphere}, Kawata's characterisation of lattices of exponent $\pi$ will hold when the ramification degree is two.

\begin{lem} \label{breakdownsucka}
Let $M,N$ be $\og$-lattices, and suppose $\exp(M)\geq \exp(N)=\pi^a$.  Let $b\geq a$ and write $M_b$ for $M\otimes\O/\pi^b\O$, $\O_bG$ for $(\O/\pi^b\O) G$ and so on. Then the map $-\otimes \O/\pi^b\O:\shom(M,N)\r\shom(M_b,N_b)$ induced by reduction mod $\pi^b$ is injective.
\end{lem}
\begin{proof}
We first remark that our assumption implies $\pi^b\cdot\shom(M,N)=0$.  For $f\in \shom(M,N)$, suppose $f\otimes \O/\pi^b\O\in \shom(M_b,N_b)$ is zero.  Then $f\otimes \O/\pi^b\O=\alpha\cdot \beta$, for some $\beta\in\hom_{\O_bG}(M_b,P)$ and $\alpha\in\hom_{\O_bG}(P,N_b)$, where $P$ is a projective $\O_bG$ module.  Then $\alpha$ and $\beta$ lift (non-uniquely) to $\alpha'\in\hom_{\og}(M,Q)$ and $\beta'\in\hom_{\og}(Q,N)$ respectively, where $Q_b=P$.  It follows we may write $f+\pi^b g=\alpha' \cdot \beta'$ for some $g\in \hom_{\og}(M,N)$.  But as we've seen, $\pi^b g$ is  projective, and hence so was $f$.
\end{proof}

The following lemma is a special case of \cite[Proposition 4.5]{kaw2}, but we rephrase the proof in terms of the functor $R$ for convenience.

\begin{lem} \label{exp2}
Let $M$ be an indecomposable $\og$-module that is not a summand of $RN$ for any $kG$-module $N$.  Then the almost split sequence starting with $M$ is split modulo $\pi$.
\end{lem}
\begin{proof}
For an $\og$-lattice $X$, let $\alpha_X:X\r\overline{X}$ denote the canonical surjection.  Then $\alpha_M$ factors through $R\overline{M}$, say via $\varphi:M\r R\overline{M}$.  We are left with the following picture:
\[
\xymatrix{
M\ar@{=}[d] \ar[r]^\varphi & R\overline{M} \ar[d]^{\alpha_{RM}} \\
M \ar[r]^(0.4){\alpha_{RM}\varphi} \ar[d]^{\alpha_M} & R\overline{M}\otimes k \ar[d]^{\ep_M\otimes k} \\
\overline{M} \ar@{=}[r] \ar[ur]^\beta & \overline{M},
}
\]where both squares commute and $\alpha_{RM}\varphi=\beta\alpha_M$. It follows that $\beta=\varphi\otimes k$ is a split monomorphism.  Now, if $h:M\r C$ is a left almost split morphism, $h$ factors through $\varphi$ since $M$ is not a summand of $R\overline{M}$.  Since $\varphi$ split mod $(\pi)$, so too is $h$. 
\end{proof}
\begin{rmk}
Implicit in this result is that if $M$ is an indecomposable non-projective $\og$-lattice, then $\overline{M}$ is a summand of $R\overline{M}\otimes k$, and that this is independent of the ramification of $p$ in $\O$.  This is also proven more directly in \cite[Theorem 2.1 (b)]{jones} (see also Lemma \ref{j-lad}).  We will see in Proposition \ref{rk} that more generally, the projective-free $kG$-module $N$ is always a summand of $\overline{RN}$, but this requires ramification of $p$ in $\O$.
\end{rmk}

We can now prove
\begin{prop} \label{kablah}
An indecomposable $\og$-lattice $X$ has exponent $\pi$ if and only if $X$ is a summand of $RN$ for some $kG$-module $N$.
\end{prop}
\begin{proof}
If $M$ is not a summand of some $RN$ with $N$ a $kG$-module, then by Lemma \ref{exp2}, any almost projective map in $\underline{\End}(M)$ is projective modulo $(\pi)$.  By Lemma \ref{breakdownsucka}, if $X$ has exponent $\pi$ it follows that no stable endomorphism of $X$ is projective modulo $(\pi)$.  Hence $X$ is a summand of some Heller lattice.  The converse is clear.
\end{proof}

\subsection{Heller lattices of simple modules}

We study here the the lattices $\oo S=\rad P(S)$, where $S$ is a simple $kG$-module.  We give criteria for $\oo S$ to be decomposable, and relate this to question of decomposability of middle terms of almost split sequences for Kn\"{o}rr lattices of exponent $\pi$.  We start with some easy lemmas. 

\begin{lem} \label{multisoc}
Let $S$ be a non-projective simple $kG$-module.  Then $P(RS)\cong P(S)\oplus P(S)$.  
\end{lem}
\begin{proof}
We first show that $P(RS)$ is decomposable.  Since $RS$ is a quotient of $P(RS)$, and there is a non-zero map $RS\r S$, it follows that $P(S)$ is a summand of $P(RS)$.  Since our convention dictates that $P(RS)$ is the $\O$-injective hull of $\oo S$, and because $\mbox{Rank}_{\O}(\oo S)=\mbox{Rank}_{\O}(P(S))$, the fact that $\oo S$ is non-projective shows that there is some other summand of $P(RS)$.  Now, consider the exact sequence $0\r\oo S\r P(S) \r S \r 0$.  On tensoring with $k$, we have the exact sequence
\[0\r S\stackrel{\varphi}{\r} \overline{\oo S} \stackrel{\psi}{\r} \overline{P(S)} \r S \r 0,\]
where exactness on the left is due to $P(S)$ being flat as an $\O$-module and the isomorphism $\mbox{Tor}_1^{\O}(k,S)\cong S$.  It follows that $\im(\psi)\cong\ok S$, and thus $\overline{\oo S}$ is an extension of $\ok S$ by $S$.  Since  $\overline{P(RS)}$ is isomorphic to the injective hull of $\overline{\oo S}$ and is decomposable, it follows that $\soc(\overline{\oo S})=S\oplus S$, and that $P(RS)\cong P(S) \oplus P(S)$.
\end{proof}

\begin{lem} \label{rhead}
Let $S,T$ be non-projective simple modules.  Then 
\[\dim_k\hom_{\sG}(RS,RT)=\left\{\begin{array}{ll} 2 & \mbox{if} \ \ \ S=T \\ 0 & \mbox{otherwise.} \end{array}\right.\]
\end{lem}
\begin{proof}
This is immediate, since Lemma \ref{multisoc} shows that $\mbox{Head}(RS)\cong S\oplus S$.
\end{proof}

It follows from this lemma that if $RS$ is decomposable for some simple module $S$, say $RS\cong B'\oplus B''$, then $\send(B')=\send(B'')=k$.  By the characterisation of Kn\"{o}rr lattices given by Carlson and Jones in \cite[Proposition 4.3]{C-J}, then in $M$ is a Kn\"{o}rr lattice of exponent $\pi$, then $\send(M)=k$.  The main result of this section will show that  

For a simple module $S$, let $\Xi(S)$ be number of distinct indecomposable summands of $P(S)\otimes K$ (that is, the number of distinct irreducible characters $P(S)$ affords).
\begin{lem} \label{basdass}
If $\Xi(S)>2$, $RS$ is indecomposable.
\end{lem}
\begin{proof}
Suppose $RS$ is decomposable, so that we may write $RS=M_1\oplus M_2$ by Proposition \ref{rhead} (where $M_i$ has simple head isomorphic to $S$).  Let $\chi$ be a simple character appearing in $P(S)\otimes K$.  By \cite{thomp}[Theorem 1], there is an $\O$-form $B$ of $\chi$ with indecomposable head $S$.  Hence we have a non-projective map $B\r S$, and hence a map $B\r RS$.  Since $B$ is indecomposable, any non-projective map $B\r M_i$ is an isomorphism by Proposition \ref{arl}.  Without loss of generality, we have $B\cong M_1$.  Similarly, if $\chi'$ is distinct from $\chi$, and 
$B'$ is an $\O$-form of $\chi'$, then $B'\cong M_2$.  Thus $\Xi(S)\leq 2$.      
\end{proof}

\begin{lem} \label{vertexp}
Let $M$ be an $\og$-lattice which is projective relative to some subgroup $H\leq G$.  Then $\exp(M)=\exp(M\d_H)$.  In particular, if $Q$ is a vertex of $M$, $\exp(M)\leq\pi^{\nu(|Q|)}$.
\end{lem}
\begin{proof}
By \cite[Lemma 3.1]{C-J}, $\exp(M\d_{H'})\leq \exp(M)$ for all subgroups $H'\leq G$, so we just to to verify the reverse inequality.  If $M$ is projective relative to $H$, then the identity in $\End_{\og}(M)$ satisfies $Id_M=\mbox{Tr}_{H,G}(a)$ for some $a\in\End_{\O H}(M\d_H)$.  If $\exp(M\d_H)=\pi^n$ it follows that $\pi^n\cdot Id_M$ is a projective map by the $\O$-linearity of $\mbox{Tr}_{H,G}$ and the fact that the relative trace map sends projective maps to projective maps.  Thus $\exp(M)\leq\exp(M\d_H)$ as required.  Since the vertices of $M$ are the minimal subgroups of $G$ that $M$ is projective relative to, the final statement is clear.
\end{proof}

We can now show how restrictive the existence of a Kn\"{o}rr lattice of exponent $\pi$ is.  Recall first that if $X$ is an irreducible lattice (we say a lattice is irreducible if $K\otimes X$ is indecomposable) in a block $B$ of defect $d$, then the \emph{height} of $X$ is the non-negative integer $h$ satisfying $\mbox{rank}_{\O}(X)=p^{a-d+h}$, where $p^a$ is the order of a Sylow $p$-subgroup of $G$.  It is immediate from the definition of defect of a block that there is always an irreducible lattice in $B$ of height zero.  

\begin{prop} \label{tooeasy}
Let $M$ be an $\og$-lattice with $\send(M)=k$ belonging to the block $B$.  Then every indecomposable non-projective lattice in $B$ has stable endomorphism ring isomorphic to $k$.  Further, the defect group of $B$ has order $p$ and $(p)=J(\O)$.
\end{prop}
\begin{proof}
Let $S$ be a simple module in the head of $M$.  Then there is a non-projective map $M\r RS$, and by Auslander-Reiten duality, a map $RS\r M$.  Thus $M$ is a summand of $RS$ (note that $M\not\cong RS$ since $\send(RS)$ is two dimensional), and hence $RS\simeq N'\oplus N''$, where $\send(N')=\send(N'')=k$. Let us show that $\rad(P(S))\cong \oo N'\oplus \oo N''$;  we only need to prove that $\rad(P(S))$ is projective-free.  First note that $N'$ is irreducible, since otherwise there would be some non-projective lattice $L$ as a quotient, and the canonical map $N'\r L$ would have to be a split epimorphism as it is not projective, a contradiction.  The same holds for $N''$.   Since $RS$ is decomposable, $\Xi(S)\leq 2$ by Lemma \ref{basdass}, so we can write $K\otimes P(S)=mX_1\oplus nX_2$ (where $m,n$ are the respective multiplicities of the $X_i$ as summands of $K\otimes P(S)$).  Then there is an $\O$-free (hence not weakly injective) quotient $B$ of $P(S)$ affording the same character as $mX_1$.  Since $B$ is a quotient of $P(S)$, there is a non-projective map $f:B\r RS$.  $B$ is indecomposable, so $f$ is a split monomorphism and therefore $m=1$.  Similarly, $n=1$.  It follows that any projective submodule of $\rad(P(S))$ not isomorphic to $P(S)$ (which necessarily has lower $\O$-rank than $P(S)$) is irreducible, which is a contradiction since $B$ has non-zero defect.  Thus $\rad(P(S))\cong \oo N'\oplus\oo N''$ is projective-free.  Let $T$ be some simple module in $\rad(\overline{P(S)})/\rad^2(\overline{P(S)})\cong\he(\ok S)$.  The exact sequence $0\r S \r \overline{\rad(P(S))}\r \overline{P(S)}\r S \r 0$ shows that $\ok S$, and hence $T$, is a quotient of $\rad(P(S))$.  This induces a non-projective $\rad(P(S))\r RT$, and it follows without loss of generality that $N'$ is a summand of $RT$, so that $RT$ is decomposable.  We may repeat this argument for each simple module in $\rad(\overline{P(S)})/\rad^2(\overline{P(S)})$, and consequently for each simple module in higher Loewy layers of $\overline{P(S)}$, to show that $RS'$ is decomposable for every composition factor $S'$ of $\overline{P(S)}$.  It follows that $\rad(P(S))$ is decomposable for every simple module in $B$.  Therefore, the only non-projective indecomposable $\og$-lattices appearing in $B$ are the summands of the radicals of the indecomposable projective modules in $B$.  In particular, every $B$-lattice has exponent at most $\pi$.  If $D$ is the defect group of $B$, and $X$ is an irreducible lattice of height zero, then by \cite[Lemma 2.2]{gm} $X$ has exponent $\pi^{\nu(|D|)}$. It follows that $|D|=p$ and $\nu(p)=1$.    
\end{proof}

The following theorem should be compared to \cite[Theorem 2.1]{jkm}, which uses an unramified coefficient ring in its hypothesis, and comes to the same conclusion.

\begin{thm}
Let $\O$ be a complete discrete valuation ring with uniformizer $\pi$, and let $B$ be a $p$-block of $\og$ with defect group $D$.  Then $B$ contains a Kn\"{o}rr lattice of exponent $\pi$ if and only if $|D|=p$ and $\O$ is unramified.  Furthermore, the existence of one Kn\"{o}rr lattice of exponent $\pi$ implies that every indecomposable non-projective lattice in $B$ is a Kn\"{o}rr lattice of exponent $\pi$.
\end{thm}
\begin{proof}
The `if' direction and final assertion is Proposition \ref{tooeasy}.  For the converse, if $\O$ is unramified and $|D|=p$, then every non-projective lattice in $B$ has exponent $\pi$ by Lemma \ref{vertexp}, and hence any irreducible lattice in $B$ (which always exist) is a Kn\"{o}rr lattice of exponent $\pi$. 
\end{proof}

\begin{prop} \label{sky}
Let $M$ be a Kn\"{o}rr lattice of exponent $\pi$.  Then the middle term of the almost split sequence terminating in $M$ is indecomposable.
\end{prop}
\begin{proof}
It follows by the proof of Proposition \ref{tooeasy} that $M$ is a proper summand of $RS$ for some simple module $S$, and hence $\he(M)\cong S$ by Lemma \ref{rhead}.  Since $\send(M)\cong k$, it follows that a projective cover $P(S)\r M$ is right almost split.  
\end{proof}

\subsection{Almost split sequences for Kn\"{o}rr lattices}

We now turn to Kn\"{o}rr lattices of exponent greater than $\pi$.  Before proceeding we introduce some convenient notation.  We write $\O_b$ for the residue ring $\O/\pi^b\O$; for an $\og$-lattice $M$ we write $M_b$ for the residue module $M/\pi^bM$, for a homomorphism $\varphi:M\r N$ between $\og$-lattices we write $\varphi_b$ denotes the morphism $\varphi\otimes\O_b$, and we write $\Omega_b$ for the Heller operator on $\O_bG$-modules.

In \cite{jones}, A. Jones $\og$-lattices of the form $\oo(M_j)$, where $M$ is an $\og$-lattice.  Such a lattices are constructed as the pullback of $M\stackrel{\pi^jId_M}{\r}M$ along a projective cover $P\r M$, as per the following commutative diagram with exact rows and columns:
\begin{equation} \label{bjdiag}
\xymatrix{
          &                            &     0\ar[d]                           &         0 \ar[d]          &  \\
 0 \ar[r] & (\oo M) \ar[r]\ar@{=}[d] & \oo(M_j) \ar[r]\ar[d]                      & M \ar[r] \ar[d]^{\pi^jId_M}      & 0 \\
 0 \ar[r] & (\oo M) \ar[r] & P \ar[r]\ar[d]                      & M \ar[r] \ar[d] & 0 \\
          &                            &   M_j \ar[d]\ar@{=}[r]                &  M_j \ar[d]                & \\
          &                            &      0                                &            0               &
}
\end{equation}


Jones showed the following, and we include a sketch proof for convenience:
\begin{lem}\label{j-lad}\cite[Theorem 2.1(a-c)]{jones}
Let $B=\oo(M_j)$ for some non-projective $\og$-lattice $M$ whose exponent is $\pi^a$.  Then $B_b\cong M_b\oplus \oo(M)_b$ for all $b$, $\exp(B)=\pi^j$ for $1\leq j\leq a$, and $B=M\oplus\oo M$ for $j> a$. 
\end{lem}
\begin{proof}
Most of this is obvious given the definition of exponent and the diagram \eqref{bjdiag}.  If $j> a$, the map $\pi^jId_M$ is projective, and hence the top row of \eqref{bjdiag} is split.  Moreover, for all $j$, the map $\pi^jId_M\otimes\O_j$ is projective, hence the top row of \eqref{bjdiag} splits modulo $(\pi^j)$.  
\end{proof}

When $M$ is projective-free, it follows from Lemma \ref{j-lad} that $B$ is also projective free.  Therefore, $B\cong R(\Omega_j(M_j))$ in the module category, and we will persist with this identification.  If $M$ is additionally a Kn\"{o}rr lattice, it is shown in \cite{C-J} that $\oo M$ is also a Kn\"{o}rr lattice (this is an easy consequence of the definition of Property E).  Hence if $M$ has exponent $\pi^a$, then $R(\Omega_{a-1}M_{a-1})$ is the middle term of the almost split sequence terminating in $M$.  We will given necessary and sufficient conditions (Theorem \ref{mdterm}) for $R(\Omega_{a-1}M_{a-1})$ to be indecomposable, based on the decomposability of $M_{a-1}$.  Geline and Mazza (\cite[Theorem 1.1]{gm}) show that the middle term of the almost split sequence terminating in the Kn\"{o}rr lattice $M$ provided $M/\pi M$ is indecomposable and $\exp(M)\geq 2$, and so Theorem \ref{mdterm} represents an extension of this result.  Our theorem follows from a much more general result, Proposition \ref{aindec}, which gives an indecomposability criterion for the lattices $B$ appearing in Lemma \ref{j-lad}.

Let us remark on the naturality of the splitting of $B_b$ in Lemma \ref{j-lad}.  Consider the following diagram, which is induced by tensoring the pullback diagram \eqref{bjdiag} by $\O_b$:

\[
\xymatrix{
          &                            &     0\ar[d]                           &         0 \ar[d]          &  \\
          &                            & \tor_1^{\O}(M, \O_b) \ar[d] \ar@{=}[r] & M_b  \ar[d]^{\simeq} &  \\  
 0 \ar[r] & (\oo M)_b \ar[r]\ar@{=}[d] & B_b \ar[r]\ar[d]                      & M_b \ar[r] \ar[d]^{\pi^bId_M\otimes \O_b}      & 0 \\
 0 \ar[r] & (\oo M)_b \ar[r] & P_b \ar[r]\ar[d]                      & M_b \ar[r] \ar[d]^{\simeq} & 0 \\
          &                            &   M_b \ar[d]\ar@{=}[r]                &  M_b \ar[d]                & \\
          &                            &      0                                &            0               &
}
\]
Since $\pi^bId_M\otimes \O_b=0$, it follows that the map $\tor_1^{\O}(M,\O_b)\r B_b$ is a splitting of $B_b\r M_b$.

\begin{lem}\label{msplaced2}
Let $M,N$ be a lattices with $\exp(N)\geq\exp(M)=\pi^a$, and let $b\leq a$.  Then the map on hom spaces induced by $R(\md_b)_b$ is given by
\begin{eqnarray*}
\shom(M_b,N_b) & \r & \shom(R(M_b)_b,R(N_b)_b) \\
\varphi & \mapsto & \begin{pmatrix} \varphi &\alpha_\varphi \\ \beta_\varphi & \Omega_b^{-1} \varphi \end{pmatrix},
\end{eqnarray*}
where the maps $\beta_\varphi$ and $\alpha_\varphi$ only depend on the class of $\varphi$ in $\sG$.
\end{lem}
\begin{proof}
Let $P\stackrel{\theta}{\r} M_b$ be an $\og$-projective cover of $M_b$, $Q\stackrel{\varpi}{\r} N_b$ and $\og$-projective cover of $N_b$.  We first show that the map $\shom(M_b,N_b)\r\shom(R(M_b),R(N_b))$ induced by $R$ is injective.  For this, it is clear that we need only check that a non-projective map $\varphi:M_b\r N_b$ induces a non-projective map $\oo(\varphi):\oo (M_b)\r \oo (N_b)$.  Consider the commutative diagram with exact rows
\[
\xymatrix{
0\ar[r] & \tor_1^{\O}(M_b,\O_b) \ar[r] \ar[d]^{\tor_1^{\O}(\varphi,\O_b)} & (\oo M)_b \ar[r] \ar[d]^{\oo(\varphi)_b} & P_b \ar[r]^{\theta_b} \ar[d] & M_b \ar[r] \ar[d]^\varphi & 0 \\
0\ar[r] & \tor_1^{\O}(N_b,\O_b) \ar[r]                                 & (\oo N)_b\ar[r]                          & Q_b \ar[r]^{\varpi_b}        & N_b \ar[r]                & 0
}
\]
induced by $\varphi$. Note that by Lemma \ref{breakdownsucka} it is enough to check that 
\[\oo(\varphi)_b=\left(\begin{array}{cc} \alpha & \beta \\ \gamma & \delta\end{array}\right):M_b\oplus \Omega_b M_b\r N_b\oplus \Omega_b N_b\]
is not projective.
The diagram above induces the following morphism of sequences
\[
\xymatrix{
0\ar[r] & \tor_1^{\O}(M_b,\O_b) \ar[r] \ar[d]^{\tor_1^{\O}(\varphi_b,\O_b)} & (\oo M)_b \ar[r] \ar[d]^{\oo(\varphi)_b} & \Omega_b M_b \ar[r] \ar[d]^{\Omega_b\varphi} & 0 \\
0\ar[r] & \tor_1^{\O}(N_b,\O_b) \ar[r]  & (\oo N)_b \ar[r]  &  \Omega_b N_b \ar[r]  & 0.
}
\]
Since $\tor_1^{\O}(-,\O_b)$ is naturally isomorphic to the identity functor on $\O_bG$-modules, it follows that this sequence is split, and hence $\alpha=\varphi$ and $\delta=\Omega_b\varphi$.  As $\varphi$ was non-projective by design, it follows that $\oo(\varphi)$ is non-projective too.

Therefore, $\oo(\varphi)_b$ is determined completely by the class of $\varphi$ in $\sG$, and the lemma follows as $\Omega_b^{-1}(\oo(\varphi)_b)=R(\varphi)_b$.
\end{proof}

Let $n=\nu(|G|)$. Maranda \cite{mar} has shown that a lattice is indecomposable if and only if it is indecomposable modulo $(\pi^{2n+1})$.  Using similar methods (see \cite[p.540]{c-r} for more details), one can show that a lattice $M$ of exponent $\pi^a$ is indecomposable if and only if $M_{a+1}$ is indecomposable.  The following shows that this is the best possible result of this type (Corollary \ref{werty} shows that all indecomposable lattices of exponent $\pi$ are decomposable modulo $(\pi)$ when $(p)\subseteq (\pi^2)$).

\begin{prop} \label{pker}
Let $M$ be an indecomposable $\og$-lattice of exponent $\pi^{a}$.  Then $M_a$ has at most two summands.  
\end{prop}
\begin{proof}
We may clearly assume $M$ is not projective.  Since $R(M_a)=M\oplus \oo^{-1}M$, it follows that the summands of $M_a$ not in the kernel of $R$ number at most two.  If $M_a$ has a weakly injective summand, then so does $\overline{M}$, and such a summand would lift to a projective summand of $M$, a contradiction.  Finally, since every summand of $M_a$ is a quotient of the projective-free lattice $M$, no summand of $M_a$ is in the kernel of $R$ by Lemma \ref{kerlem}.  
\end{proof}

\begin{prop} \label{aindec}
Let $M$ be an $\og$-lattice of exponent $\pi^a$.  Then $R(M_b)$ is indecomposable provided $M_b$ is indecomposable and $b< a$.
\end{prop}
\begin{proof}
Suppose $R(M_b)$ is decomposable.  Since $R(M_b)_b=M_b \oplus \oo^{-1}(M)_b$, it follows that $R(M_b)=N\oplus \oo^{-1}N$, where $N$ is an indecomposable lattice of exponent $\pi^b$ such that $N_b\cong M_b$.  By the proof of Lemma \ref{msplaced2}, the composition $\send(N)\r \send{N_b}\r\send{R(N_b)}$ is a monomorphism of algebras.  In particular, if $0\neq\varphi\in\soc(\send(N))$, then $R(\varphi_b)\in\send(N\oplus\oo^{-1}N)$ is a non-zero element of $\soc(\send(N\oplus\oo^{-1}N))$.  Now, since $\dim_k(\he(\send(N,\oo^{-1}N))=2$ and $\send(N\oplus\oo^{-1}N)$ is a symmetric algebra by \cite[Theorem 2.2]{thev2} over the Artinian ring $\O_n$, it follows that $\dim_k(\soc(\send(N\oplus\oo^{-1}N))=2$.  Thus it is clear that the morphisms
\[\psi:=\begin{pmatrix} \varphi & 0 \\ 0 & 0 \end{pmatrix} \ \ \mbox{and} \ \ \psi':=\begin{pmatrix} 0 & 0 \\ 0 & \oo^{-1}\varphi\end{pmatrix}\]
form a basis for the socle of $\send(N\oplus\oo^{-1}N)$.  Indeed, they are obviously linearly independent, and since $\varphi$ generates the socle of $\send(N)$, $\psi$ and $\psi'$ are annihilated by all elements of $\rad(\send(N\oplus\oo^{-1}N))$.

Now, we have a commutative diagram
\[
\xymatrix{
                        &                            & M\ar[d]^{(f',f'')}          &           \\
R(\oo(N)_b) \ar[r] \ar[d] & R(Y_b) \ar[r]\ar[d] & N\oplus\oo^{-1}(N) \ar[d]^{\ep_{M_b}} \ar[r]^(0.6){R(\varphi_b)} & R(N_b) \ar[d] \\
\oo(N)_b \ar[r] & Y_b \ar[r] & M_b \ar[r]^{\varphi_b} & M_b,
}
\]
where the maps between then triangles are the respective counits, and $(f',f''): N\oplus\oo^{-1}(N)$ is the image of the reduction map $M\r M_b$ under the adjunction isomorphism $\shom(M,M_b)\cong\shom(M,R(M_b))$.  Then $\varphi_b\circ \ep_{M_b}\circ(f',f'')$ is not projective, and hence $R(\varphi_b)\circ (f',f'')$ is also non-projective.  However, by the description of the socle of $\send(N\oplus\oo^{-1}N)$ given above, this can only happen if one of $f',f''$ is an isomorphism, but this contradicts the assumption that $\exp(M)=\pi^a>\exp(N)=\pi^b$.  The proposition follows. 
\end{proof}

\begin{thm} \label{mdterm}
Let $M$ be a Kn\"{o}rr lattice of exponent $\pi^a$.  Then the almost split sequence terminating in $M$ is indecomposable if and only if $M/\pi^{a-1}M$ is indecomposable.  Further, if $a=1$, then the middle term of the almost split sequence terminating in $M$ is indecomposable projective. 
\end{thm}
\begin{proof}
The first part follows from Propositions \ref{aindec}.  The case when $a=1$ is implied by Proposition \ref{sky}.
\end{proof}

\section{Heller lattices in ramified extensions}

We turn to Heller lattices of $kG$-modules and their indecomposability. It has been observed by many authors (for instance, \cite{jones},\cite{kaw1}) that one may find indecomposable $kG$-modules $N$ such that $\oo N$ is decomposable; for an explicit example, one may show that $J(\og)$ is decomposable when $(|G|)=J(\O)$ (see \cite[Lemma 1.1]{kaw3}), and hence $\oo k$ is decomposable for when $G$ is a cyclic group of order $p$.  In fact, in this case, Heller and Reiner \cite{HR} have shown that there are only two indecomposable non-projective $\og$ lattices, the trivial lattice $\O$ and the augmentation ideal $\mathfrak{A}$.  It follows that if $M$ is any indecomposable $kG$-module, $\oo M= \O \oplus \mathfrak{A}$.  


 Further, Example \ref{gtytg} suggests it is not even obvious the Heller lattice of a projective-free $kG$-module need necessarily be projective-free. 

 Proposition \ref{aindec} shows that $\oo (M)$ is indecomposable whenever $M$ is an indecomposable liftable $kG$-module and $(|G|)\neq J(\O)$, independent of how $p$ (the characteristic of $k$) ramifies in $\O$.  We are able to do much better when we allow $p$ to ramify in $\O$.

\subsection{Ramified $\O$}

From now on, we consider $\O$ as a (finite) purely ramified extension $(\O',\pi')\subset (\O,\pi)$ of complete discrete valuation rings of ramification degree $e$. Recall that this means that $\O$ is a free $\O'$-module of finite rank such that $(\pi')=(\pi^e)$, where $\pi'$ and $\pi$ generate the maximal ideals of $\O'$ and $\O$ respectively.  Further, the residue fields $k=\O/\pi\O$ and $k'=\O'/\pi'\O'$ are isomorphic (and we identify them).  
Arbitrary finite extension of valuation rings may written as an unramified extension (giving an extension of residue fields) followed by a purely ramified extension.  We focus only on purely ramified extensions for ease of exposition; none of the results obtained are materially affected by this convention (provided all $kG$-modules considered are definable over the appropriate subfields of $k$).  Of course, the extension above is the result of an algebraic extension of complete discrete valuation fields $K/K'$, and $\O$ and $\O'$ are the appropriate rings of integers.

Feit has shown that when considering ramified extensions of complete local rings, one gains a degree of control over the $\pi$-modular reductions of lattices. To be precise, he shows

\begin{lem}\label{feity}
Let $A'$ be an $\O'$-algebra, and $M$ a finitely generated $A$-lattice.  Let $0=V_n\subseteq V_{n-1} \subseteq \cdots \subseteq V_1=\overline{M}$ be a filtration of the $\overline{A'}$-module $\overline{M}$, with $W_i=V_i/V_{i+1}$.  Let $\O$ be an extension of $\O'$ with ramification index $e\geq n$, and set $A=\O\otimes_{\O'}A$.  Then there is an $A$-module $N$ such that $K\otimes N\cong K\otimes M$ with $\overline{N}\cong \bigoplus_{i=1}^n W_i$.
\end{lem}
\begin{proof}
This is \cite[Lemma 18.2]{Feit}.
\end{proof}

The key observation following this lemma is that given a $kG$-module $M$ we may construct an $\og$-module $N$ such that $\overline{N}\cong M\oplus \ok^{-1}M$: simply lift an injective hull of $M$ over $\O'$ to get an $\O'G$-lattice $P'$, and apply the lemma to the filtration $0\subset M \subset \overline{P'}$. Hence we observe:
\begin{lem}\label{cosurj}
Every projective-free $kG$-module is the quotient of a projective-free lattice.  In particular, the counit maps $RM\r M$ are surjective.
\end{lem}
\begin{proof}
For a projective-free $kG$-module $M$, the first part follows from Feit's Lemma applied to the $\og$-projective cover $P(M)$ of $M$ with $\Omega_k(M)\subset \overline{P(M)}$. 
For the second assertion, let $f:B\r M$ be some surjective map, with $B$ projective-free.  Consider the induced diagram:
\[
\xymatrix{
B \ar[rr]^f \ar[dd] \ar[dr]^{\pi_B} &     &  M \\
                            & \overline{B} \ar[dr]^\beta \ar[ur]^\varphi &    \\
RM\oplus P \ar[rr] &  & \overline{RM} \oplus \overline{P}, \ar[uu]^\alpha
}
\] 
where $P$ is projective and the composition $RM\oplus P\r \overline{RM}\oplus \overline{P} \r M$ is stably the counit.  Suppose for contradiction the counit is not surjective.  Then we can take $P$ to be non-zero, and assume everything commutes in $\OG$, so $\varphi=\alpha\beta$, in particular.  Then, there is a simple module $S$ in the head of $M$ such that the composition $\theta:P\stackrel{\alpha|_P}{\r} M \r S$ (the second map being the canonical surjection) is non-zero.  If $P(S)$ is the indecomposable summand of $P$ not in the kernel of $\theta$, then $\beta$ induces the non-zero composition $\overline{B}\stackrel{\beta'}{\r} P(S) \stackrel{\theta|_{P(S)}}{\r} S$.  However, this is only non-zero if $\beta'$ is surjective and thus splits.  Since $B$ is projective-free so is $\overline{B}$, giving us a contradiction.    
\end{proof}

For a $kG$-module $M$, the syzygy $\ok(M)$ can always be chosen to be projective-free.  This property does not readily lift when we consider $M$ as an $\og$-module.  Indeed, if we take $M$ to be an arbitrary $\og$-module, then in can be that $\oo(M)$ \emph{cannot} be made to projective-free; simply take $M$ to be any non-zero object in the kernel of $R$.  However, the previous lemma implies that, provided $M$ is a $kG$-module, that this property does carry over.  

\begin{prop} \label{hirsuite}
Let $P(M)$ be an $\og$-projective cover of an indecomposable non-projective $kG$-module $M$.  Then the kernel of any surjection $P(M)\r M$ is projective-free.
\end{prop}
\begin{proof}
By abuse of notation, let $\oo(M)$ be the kernel of some surjective map $P(M)\r  M$. 
Consider the following commutative diagram with exact rows:
\begin{equation} \label{square}
\xymatrix{
   0 \ar[r]       & 0 \ar[r]  \ar[d]& \oo(M) \ar[r] \ar@{=}[d]       &  Q\ar[r]  \ar[d]         & RM \ar[r]\ar[d]^{\ep_M} & 0 \\
   0 \ar[r]       & 0 \ar[r] \ar[d] & \oo(M) \ar[r]  \ar[d]          &  P(M) \ar[r]  \ar[d]         & M \ar[r]\ar@{=}[d]  & 0 \\
 0 \ar[r]  & M \ar[r] \ar[d] & \overline{\oo(M)} \ar[r]\ar@{->>}[d] &  \overline{P(M)}\ar[r] \ar@{=}[d] & M \ar[r] \ar@{=}[d]        & 0 \\
 0 \ar[r]         &    0\ar[r]      &  \ok(M) \ar[r]        &  \overline{P(M)}\ar[r] & M \ar[r]         &  0.          
          }
\end{equation}
Here, the top two rows give the defining diagram for the counit $\ep_M$, and the third row is obtained from the second by tensoring with $k$ (note that $\tor_1^{\O}(M,k)\cong M$ and $P(M)$ is flat as an $\O$-module).  The induced map between the top and bottom rows, namely 
\[
\xymatrix{
   0 \ar[r]  & \Omega(M) \ar[r] \ar[d]       &  Q\ar[r]  \ar[d]         & RM \ar[r]\ar[d]^{\ep_M} & 0 \\
  0\ar[r]      &  \ok(M) \ar[r]        &  \overline{P(M)}\ar[r] & M \ar[r]         &  0.          
          }
\]
shows that the map $\oo(M)\r \ok(M)$ is stably the counit $\ep_{\ok(M)}$ (the projective-free part of $\oo(M)$ is $R(\ok(M))$).  Write $\oo(M)=R(\ok(M))\oplus \tilde{P'}$, so that $\tilde{P'}$ is the largest projective summand of $\oo(M)$.  Then $\oo(M)\otimes k=M'\oplus P'$, where $P'$ and $M'$ lift to $\tilde{P'}$ and $R(\ok(M))$ respectively, by \cite[Lemma 1]{thomp}.  Now, the exact sequence
\[   
\xymatrix{
0\ar[r] & M \ar[r]  & M'\oplus P' \ar[r] & \ok(M) \ar[r] & 0
}
\] 
implied by the third and fourth rows of the diagram \eqref{square} is realised by the pullback diagram
\[
\xymatrix{
M\ar[r]^f \ar[d]^{f'} & M' \ar[d]^{g'} \\
P' \ar[r]^(0.4)g & \ok(M).
}
\]
It follows that $g'$ is simply $\ep_{\ok(M)}\otimes k$, and is thus surjective.  Hence $f'$ is also surjective, implying that $P'=0$, completing the proof. 
\end{proof}

\begin{prop} \label{rk}
Let $M$ be an indecomposable non-projective $kG$ module.  Then $RM\otimes k\cong M\oplus \ok^{-1}(M)$. 
\end{prop}
\begin{proof}
By Lemma \ref{feity} there is a lattice $N$ with $\overline{N}\cong M\oplus \ok^{-1}(M)$.  Let $f:\overline{N}\r M$ be a split epimorphism.  Then in $\sG$ we have a diagram
\[
\xymatrix{
N\ar[rr] \ar[dd] \ar[dr]^{-\otimes k} &                                          &  M \\
                                      & \overline{N} \ar[ur]^f \ar[dr]^{\varphi} &  \\
RM \ar[rr]^{-\otimes k}               &                     & \overline{RM}\ar[uu]_{\ep_M\otimes k}
}
\]
where everything commutes.  Since $f$ is the composition $\overline{N}\r \overline{RM}\r M$, it follows that $\ep_M\otimes k$ is split.  Since $\overline{RM}$ and $M$ are projective-free, we deduce that $\ep_M\otimes k$ is split in $\OG$.  Since the kernel of $\ep_M\otimes k$ is $\ok^{-1}(M)$, the first part of proposition follows.
\end{proof}

\begin{rmk}
This result does not hold when $\O$ is unramified.  To see this, let $p\geq 5$ be a prime, $\O=\Z_p$ the $p$-adic integers with residue field $k=\mathbb{F}_p$ and $G=C_p$ a cyclic group of order $p$.  Let $M$ be the indecomposable $kG$-module of length two (or any length other than one or $p-1$).  We have seen in this case that there are only two non-projective indecomposable $\og$-lattices, and it is not hard to see that $RM$ is isomorphic to their sum.  It follows that $\overline{RM}\cong k\oplus \ok^{-1} k$.
\end{rmk}

With respect to the decomposition $\overline{RM}=M\oplus \ok^{-1}M$, we write morphisms in $\send(\overline{RM})$ as matrices $\begin{pmatrix} \alpha & \beta \\ \gamma & \delta \end{pmatrix}$, where $\alpha\in \send(M)$, $\beta\in \shom(\ok^{-1}M,M)$, $\gamma \in \shom(M,\ok^{-1}M)$, and $\delta \in \send(\ok^{-1}M,\ok^{-1}M)$.  We also note that, our construction of $RM$ using Feit's lemma as described after Lemma \ref{feity} coincides with Kawata's construction of the Heller lattice of $\ok^{-1}M$ in \cite{kaw2}.  This justifies our use of \cite[Lemma 4.3]{kaw2} below.

\subsection{Endomorphism rings}

In this section, we determine the stable endomorphism rings of the modules $RM$ for $M$ a $kG$-module.  We deduce as corollaries the main results of \cite{kaw2} under the slightly more general hypothesis that $\O$ has ramification degree 2 over $\O'$.  First, an analogue of Lemma \ref{msplaced2}.

\begin{lem} \label{misplaced2}
Let $M,N$ be projective-free $kG$-modules.  Then the map on hom spaces induced by $R(\md)\otimes k$ is given by
\begin{eqnarray*}
\shom(M,N) & \r & \shom(\overline{RM},\overline{RN}) \\
\varphi & \mapsto & \begin{pmatrix} \varphi &\alpha_\varphi \\ \beta_\varphi & \ok^{-1} \varphi \end{pmatrix},
\end{eqnarray*}
where the maps $\beta_\varphi$ and $\alpha_\varphi$ only depend on the class of $\varphi$ in $\sG$.
\end{lem}
\begin{proof}
The proof is entirely analogous to that of Lemma \ref{msplaced2} (take $b=1$ in the proof), given the decompositions of $\overline{RM}$ and $\overline{RN}$ given by Proposition \ref{rk}.
\end{proof}
\begin{rmk}
We strongly suspect the morphisms $\alpha_\varphi$ and $\beta_\varphi$ appearing in the lemma are projective, but we have no proof of this.  It will turn out to not present a problem, however.
\end{rmk}

Let $M$ be a $kG$-module.  We have linear isomorphisms 
\[\send(RM)\cong\shom(RM,M)\cong\shom(M\oplus \ok^{-1}M,M),\]
and hence we may identify $\send(RM)$ with a summand of $\send(M\oplus \ok^{-1}M)$ in the obvious manner.

 Kawata \cite[Lemma 4.3]{kaw2} has shown that, when the ramification degree of $\O$ over $\O'$ is at least three, given a $kG$-homomorphism $\mu:\ok^{-1} M\r M$, there is an $\og$-endomorphism $\hat{\mu}$ of $RM$ such that 
\[\hat{\mu}\otimes k=\begin{pmatrix} 0 & \mu \\ 0 & 0 \end{pmatrix}\in \End_{kG}(M\oplus \ok^{-1}M).\]
The same proof works when the ramification degree is at least two by our Proposition \ref{rk}.  Further, since an endomorphism $f$ of $RM$ is projective if and only if $f\otimes k$ is projective, $\hat{\mu}$ is projective if and only if $\mu$ is projective.  Hence the morphisms $\hat{\mu}$ span a $\dim_k\shom(\ok^{-1}M,M)$-dimensional subspace of $\send(RM)$.  Lemma \ref{misplaced2}, on the other hand, determines a $\dim_k\send(M)$-dimensional subspace of $\send(RM)$.  These subspaces intersect trivially, and since $\dim_k\send(RM)=\dim_k\shom(M\oplus\ok^{-1}M,M)$, it follows that given $f\in\send(RM)$, we may write $f\otimes k\in\send(M\oplus\ok^{-1}M)$ uniquely as a sum $\begin{pmatrix}\varphi & 0 \\ \beta_\varphi & \ok^{-1}\varphi \end{pmatrix} + \begin{pmatrix} 0 & \mu \\ 0 & 0 \end{pmatrix}$ where $\beta\in\send(M)$ and $\mu\in\shom(\ok^{-1}M,M)$.  In particular, the component of $f$ mapping $M$ to $M$ and mapping $\ok^{-1}M$ to $\ok^{-1}M$ mutually determine one another (as maps in $\sG$).  

Before stating our next result, we recall a definition.  Let $\Lambda$ be a $k$-algebra, and $M$ a $(\Lambda\md\Lambda)$-bimodule.  The \emph{trivial extension of $\Lambda$ by $M$}, $T_M\Lambda$, is the $k$-algebra whose underlying vector space is simply the direct sum $\Lambda\oplus M$, and the multiplication in $T_M\Lambda$ is given by $(\lambda, m)(\lambda',m')=(\lambda\lambda',\lambda\cdot m'+m\cdot \lambda')$, where $\cdot$ denotes the appropriate actions of $\Lambda$ on $M$.  In the special case when $M$ is the $k$-linear dual of $\Lambda$, we simply write $T\Lambda$ instead of $T_M\Lambda$, and refer to this ring as the trivial extension algebra of $\Lambda$.

\begin{thm}\label{uphere}
Let $G$ be a group of order divisible by the rational prime $p$, let $\O'\subset\O$ be a ramified extension of complete discrete valuation rings, and let $M$ be a $kG$-module.  Then $\send(RM)\cong T\send(M)$.
\end{thm}
\begin{proof}
By the Auslander-Reiten duality for $kG$-modules, $\send(\ok^{-1}M,M)$ is naturally isomorphic to the dual of $\send(M)$ as an $(\send(M)\md\send(M))$-bimodule.  Here,  if $\varphi\in\send(M)$ and $\mu\in\send(\ok^{-1}M,M)$, then the left action of $\varphi$ on $\mu$ is just the composition $\varphi\mu$, whereas the right action is the composition $\mu\ok^{-1}\varphi$.  Hence $\send(RM)$ has the appropriate underlying vector space.  Let $\Phi:\send(RM)\r T\send(M)$ be the corresponding linear isomorphism, so that $\Phi(f)=(\varphi,\mu)$, where we have $f\otimes k=\begin{pmatrix} \varphi & \mu \\ \beta_\varphi & \ok^{-1}\varphi \end{pmatrix}$ (c.f. Lemma \ref{misplaced2}).  Let $f=\Phi^{-1}(\varphi,\mu)$ and $f_1=\Phi^{-1}(0,\mu')$.  Then 
\[(ff_1)\otimes k=\begin{pmatrix} 0 & \varphi\mu' \\ 0 & \beta_\varphi\mu' \end{pmatrix} \mbox{\ \ \ and \ \ \ } (f_1f)\otimes k=\begin{pmatrix} \mu'\beta_\varphi & \mu'\ok^{-1}\varphi \\ 0 & 0 \end{pmatrix}.\]
It follows in particular that $\beta_\varphi\mu'$ and $\mu'\beta_\varphi$ are projective (as $kG$-homomorphisms) for all $\mu'$ and all $\beta_\varphi$ by the preceding discussion.  Now let $f'=\Phi^{-1}(\varphi',\mu')$.  Then, as maps in $\sG$,
\begin{eqnarray*}
\Phi(ff') & = &(\varphi\varphi'+\beta_{\varphi'}\mu, \varphi \mu' +\mu\ok^{-1}(\varphi'))\\
          & = & (\varphi\varphi', \varphi \mu' +\mu\ok^{-1}(\varphi'))\\
          & = &(\varphi,\mu)(\varphi',\mu') \\
          & = & \Phi(f)\Phi(f'),
\end{eqnarray*}
and $\Phi$ is an algebra isomorphism.
\end{proof}
\begin{rmk}
We also believe the analogue of this result holds for the lattices $R(M_b)$ considered in Proposition \ref{aindec}.  The only additional result required for a proof is an analogue of Kawata's lemma \cite[Lemma 4.3]{kaw2}.  We suspect such an analogue can be deduced using just the adjunction $R$ (so we do not even need Kawata's proof for Theorem \ref{uphere}). 
\end{rmk}
We recover the main results of \cite{kaw2}.
\begin{cor} \label{werty}
Let $M$ be an indecomposable $kG$-module.  Then $\oo M$ is indecomposable.
\end{cor}
\begin{proof}
Recall that $\oo M$ is projective-free by Proposition \ref{hirsuite}.  Since trivial extension algebras of local algebras are themselves local, the result follows from Theorem \ref{uphere} since $\send(RM)\cong\send(\oo M)$.
\end{proof}
\begin{cor}
Let $\mathcal{A}(M):0\r \oo^2 M \r B \r \oo M\r0$ be the almost split sequence terminating in the Heller lattice $\oo M$ of an indecomposable non-projective $kG$-module $M$.  Then the reduced sequence $\mathcal{A}_k(M):0\r \oo^2 M\otimes k \r B\otimes k \r \oo M\otimes k\r 0$ is the sum of the almost split sequence terminating in $M$ and the split sequence $0\r \ok M\r \ok M\oplus \ok M \r \ok M\r 0$.
\end{cor}
\begin{proof}
Using the notation from Theorem \ref{uphere} and its proof, it is clear that the socle of $\send(\oo M)$ is generated by $(0,\mu)$, where $\mu$ generates the simple $\send(M)$-socle of $\shom(M,\ok M)$.  Then $\mathcal{A}_k(M)$ is the pullback of $(0,\mu)\otimes k=\begin{pmatrix} 0 & \mu \\ 0 & 0 \end{pmatrix}\in \send(\ok M\oplus  M)$ along a projective cover $P\r \ok M \oplus M$, and the result follows.
\end{proof}

\large{University of Bristol, University Walk, Bristol, BS8 1TW, United Kingdom}

Email address: \tt{Andrew.Poulton@bristol.ac.uk}


\begin{thebibliography}{99}



\bibitem{BIK} 
D. J. Benson, S. Iyengar and H. Krause, \emph{Module categories for group algebras}; arXiv:1208.1291v1; 2012.





\bibitem{C-J}
J. F. Carlson and A. Jones, \emph{An exponential property of lattices over group rings}, J. London Math Soc. (2) 39; 1989.

\bibitem{c-r}
C. Curtis and I. Reiner, \emph{Representation theory of finite groups and associative algebras}, Wiley Classics Library; 1988.

\bibitem{Feit} 
Walter Feit, \emph{Representation Theory of Finite Groups}, North-Holland Mathematical Library Volume 25, North-Holland; 1982.


\bibitem{gm}
M. Geline and N. Mazza, \emph{Auslander-Reiten components for some Kn\"{o}rr lattices}, Bull. London Math. Soc. 45; 2013.

\bibitem{hap}
Dieter Happel, \emph{Triangulated categories in the representation theory of finite dimensional algebras}, LMS Lecture Note Series 119, CUP; 1988.

\bibitem{HR}
A. Heller and I. Reiner, \emph{Representations of cyclic groups in rings of integers}, Annals of Mathematics, Vol. 76, No. 1; 1962.


\bibitem{jones}
Alfredo Jones, \emph{On the exponent of lattices over group rings}, An. St. Univ. Ovidius Constantza 4; 1996.

\bibitem{jkm}
A. Jones, S. Kawata, G. Michler, \emph{On exponents and Auslander-Reiten components of irreducible lattices}, Arch. Math. 76; 2001.

\bibitem{kaw3}
S.Kawata, \emph{On Auslander-Reiten components and projective lattices of $p$-groups}, Osaka J. Math 38; 2001.

\bibitem{kaw1}
S. Kawata, \emph{On Auslander-Reiten components and Heller lattices for integral group rings}, Algebr Represent Theor 9; 2006.

\bibitem{kaw2}
S. Kawata, \emph{On Heller lattices over ramified extended orders}, Journal of Pure and Applied Algebra 202; 2005. Erratum: Journal of Pure and Applied Algebra 212 pp. 1849-1851; 2008.

\bibitem{kaw4}
S.Kawata, \emph{On Auslander-Reiten components and splitting trace lattices for integral group rings}, Jorunal of Algebra 359; 2012.


\bibitem{knorr}
R. Kn\"{o}rr, \emph{Virtually irreducible lattices}, Proc. London Math Soc.(3) 59; 1989.

\bibitem{mar}
J. M. Maranda, \emph{On $p$-adic integral representations of finite groups}, Can. J. Math. 5; 1953.





\bibitem{qui}
D. G. Quillen, \emph{Higher algebraic $K$-theory I}, Algebraic $K$-theory I: Higher $K$-theories, Lecture Notes in Mathematics, vol. 341, Springer-Verlag, Berlin/New York; 1973.

\bibitem{rik}
Jeremy Rickard, \emph{Derived categories and stable equivalences}, Journal of Pure and Applied Algebra 61; 1989.


\bibitem{tev}
Jacques Th\'evenaz, \emph{$G$-Algebras and modular representation theory}, OUP; 1995.

\bibitem{thev2}
Jacques th\'evenaz, \emph{Duality in $G$-algebras}, Math. Z. 200; 1988.

\bibitem{thoma}
R. W. Thomason, \emph{The classification of triangulated subcategories}, Compositio Mathematica 105; 1997.

\bibitem{thomp}
John G. Thompson, \emph{Vertices and sources}, Journal of Algebra Vol. 6, Issue 1; 1967.
\end{thebibliography}
\end{document}